\documentclass[12pt,a4paper]{article}
\usepackage{graphicx}
\usepackage{amsmath,amsfonts,amsthm}
\usepackage{xfrac,esint,mathrsfs}
\usepackage[utf8]{inputenc}
\newcommand\eps{\varepsilon}

\newcommand\Lip{\mathrm{Lip\,}}

\newcommand\R{\mathbb{R}}
\newcommand\N{\mathbb{N}}

\newcommand\calH{\mathcal{H}}
\newcommand\calM{\mathcal{M}}

\newcommand\calL{\mathcal{L}}
\newcommand{\LM}[1]{\hbox{\vrule width.2pt \vbox to#1pt{\vfill \hrule width#1pt height.2pt}}}
\newcommand{\LL}{{\mathchoice{\,\LM7\,}{\,\LM7\,}{\,\LM5\,}{\,\LM{3.35}\,}}}
\newtheorem{theorem}{Theorem}[section]
\newtheorem{definition}[theorem]{Definition}

\newtheorem{lemma}[theorem]{Lemma}
\newtheorem{remark}[theorem]{Remark}
\newtheorem{example}[theorem]{Example}

\numberwithin{equation}{section}
\newcounter{Nummer}

\usepackage{fancyhdr}
\newcommand{\A}{A}
\newcommand\res{\mathop{\hbox{\vrule height 7pt width .5pt depth 0pt
\vrule height .5pt width 6pt depth 0pt}}\nolimits}

\newcommand{\Hn}{{\mathcal H}^{n-1}}
\newcommand{\Rn}{{\R}^n}
\newcommand{\ba}[1]{\begin{eqnarray} #1 \end{eqnarray}}
\newcommand{\be}[1]{\begin{equation} #1 \end{equation}}

\newcommand{\bes}[1]{\begin{eqnarray*} #1 \end{eqnarray*}}

\newcommand{\Ln}{{\mathcal L}^n}

\begin{document}
\begin{center}
  {\Large
Which special functions of bounded deformation\\[2mm]
have bounded variation?}\\[5mm]
{\today}\\[5mm]
Sergio Conti$^{1}$, Matteo Focardi$^{2}$, and Flaviana Iurlano$^{1}$\\[2mm]
{\em $^{1}$
 Institut f\"ur Angewandte Mathematik,
Universit\"at Bonn\\ 53115 Bonn, Germany}\\[1mm]
{\em $^{2}$ DiMaI, Universit\`a di Firenze\\ 50134 Firenze, Italy}\\[3mm]
\begin{minipage}[c]{0.85\textwidth}
   {\bf Abstract.} Functions of bounded deformation ($BD$) arise naturally in the study of fracture and damage
    in a geometrically linear context. They are related to functions of bounded
    variation ($BV$), but are less well understood. We discuss here the 
    relation to $BV$ under additional regularity assumptions, which may require 
    the regular part of the strain to have higher integrability or the jump set to have finite area or the Cantor
    part to vanish. 
    On the positive side, we prove that $BD$ functions which are piecewise
    affine on a Caccioppoli partition are in $GSBV$, and we prove that 
     $SBD^p$ functions are approximately continuous $\calH^{n-1}$-a.e. away from the jump set. 
     On the negative side, we construct a function which is $BD$ but not in $BV$ and has distributional 
     strain consisting only of a jump part, and one which has a distributional strain consisting of 
      only a Cantor part.     
    \end{minipage}
\end{center}

\section{Introduction}

The space $BD(\Omega)$ of functions of bounded deformation  is 
characterized by the fact that the symmetric part of the distributional gradient
$Eu:=(Du+Du^T)/2$ is a bounded Radon measure, 
\begin{equation}
 BD(\Omega):=\{u\in  L^1(\Omega;\R^n): Eu \in \calM(\Omega;\R^{n\times n}_\mathrm{sym})\},
\end{equation}
where $\Omega\subseteq\R^n$ is an open set.
$BD$ constitutes the natural setting
for the study of plasticity,
damage and fracture models in a geometrically linear framework
\cite{Suquet1978a,Temam1983,TemamStrang1980,AnzellottiGiaquinta1980,KohnTemam1983}.
Despite its importance, it is not yet completely understood. 

One crucial property of $BD$ functions is that the strain can be decomposed in a part absolutely continuous
with respect to the Lebesgue measure $\calL^n$, a jump part and a third part, called Cantor part,
\begin{equation}\label{eqEudecommp}
 Eu =e(u) \calL^n + \frac{[u]\otimes \nu+\nu\otimes [u]}{2} \calH^{n-1}\LL J_u + E^cu\,.
\end{equation}
Here $J_u$ is the jump set, which is a $(n-1)$-rectifiable subset of $\Omega$, $[u]:J_u\to\R^n$ is the jump 
of $u$, and $\nu$ the normal to $J_u$, see \cite{amb-cos-dal} for details. This decomposition
is very similar to the one holding for the space of functions of bounded
variation ($BV$),
which is defined as the set of $L^1$ functions whose gradient is a bounded measure
\cite{ambrosio},
but $BD$ is strictly larger than $BV$.
A function such that the symmetric part of the gradient is integrable, but the
full gradient is not integrable, was first constructed by Ornstein in 1962
\cite{ornstein}, a simpler construction was  obtained in \cite{ContiFaracoMaggi2005}
using laminates with unbounded support, a much more general statement was then 
proven in  \cite{KirchheimKristensen2011}.  In all these examples only the first term in 
(\ref{eqEudecommp}) is nonzero. It is therefore natural to ask whether a function $u\in BD(\Omega)$ such that
 $e(u)$ has higher integrability may be in $BV(\Omega)$. 
 We address this question in Section~\ref{sbd-purejump} below.
 We remark that 
 the answer would  obviously be positive 
 if instead $Eu=e(u)\calL^n$ with $e(u)\in L^p$, since in this case $u\in W^{1,p}$ by Korn's inequality.

In the modeling of fracture in linear elasticity
one often focuses on  the set of special functions of 
bounded deformation 
\begin{equation*}
 SBD(\Omega):=\{u\in BD(\Omega): E^cu=0\}\,,
\end{equation*}
see for example \cite{FrancfortMarigo1998,Chambolle2003,SchmidtFraternaliOrtiz2008,BourdinFrancfortMarigo2008,FocardiIurlano2014,Iurlano2014}; here $e(u)$ is interpreted as an elastic strain, and the jump part as a fracture term.
The space $SBD$ is however only closed with respect to topologies that entail a bound stronger
than $L^1$ on the regular part $e(u)$ of the strain, and a constraint on the $\calH^{n-1}$ measure of the jump set.
Therefore one naturally considers the subspace $SBD^p$, defined for $p\in (1,\infty)$ as
\begin{equation}\label{eqefsbdp}
 SBD^p(\Omega):=\{u\in BD(\Omega): E^cu=0,\hskip1mm e(u)\in L^p(\Omega),\hskip1mm \calH^{n-1}(J_u)<\infty\}\,,
\end{equation}
 see for example
\cite{BellettiniCosciaDalmaso1998,Chambolle2004a,Chambolle2004b}. 

The analysis of models based on $BD$ and $SBD^p$ functions often requires
knowledge of their fine properties,
which are still not well understood. In contrast, in the $BV$ framework
fine properties are known in much more detail \cite{ambrosio}, and have proven
useful for example in studying relaxation \cite{bou-bra-but,bou-fon-masc,bou-fon-leo-masc}.
Correspondingly, the study of lower semicontinuity and relaxation in $BD$ is still 
in its beginnings. Lower semicontinuity was studied in $SBD^p$ 
by Bellettini, Coscia and Dal Maso 
\cite{BellettiniCosciaDalmaso1998}
(see \cite{GargZap11} for further results on surface integrals and related references),
 and more recently in  $BD$  by Rindler \cite{Rindler2011} 
(see \cite{EbobisseLD} for partial results).
Relaxation in $BD$ was studied
in \cite{BraidesDefranceschiVitali1997,BraidesDefranceschiVitali2002,ContiOrtiz05} for specific plasticity models,
in \cite{BarFonToa00} for autonomous functionals with linear growth defined on $W^{1,1}$,
an extension to the full space $BD$ is provided in \cite{Rindler2011}. 
General integral representation result have been established for certain functionals with linear growth 
restricted to $SBD$ in \cite{EboToa03}. 
Compactness and approximation results in $SBD^p$ with more regular functions have been 
obtained in \cite{BellettiniCosciaDalmaso1998,Chambolle2004a,Chambolle2004b,Iurlano2014}. 

It is therefore important to 
 understand the relation between $BD$, $SBD^p$ and $BV$. 
In this paper, we contribute to this topic in studying three different problems.

First, we show that if $u\in SBD(\Omega)$ is a piecewise rigid displacement, in the sense that 
the strain $Eu$ only consists of a jump part and the total length of the 
jumps is finite, then $u\in GSBV(\Omega;\R^n)$, see Theorem \ref{theocacciopp} for a precise formulation.
In particular, if $u$ is bounded then $u\in SBV(\Omega;\R^n)$.
The class of piecewise rigid displacements arises naturally in analyzing rigidity properties in the 
framework of linearly 
elastic fracture mechanics (see, for example, \cite{cha-gia-pon}).

Secondly, we construct a function in $SBD(\Omega)\setminus GBV(\Omega;\R^n)$ which has $e(u)=0$ almost everywhere 
(Theorem~\ref{t:purejump}). 
This is based on a modification of the construction from \cite{ContiFaracoMaggi2005}.  A variant of our 
construction leads to a function in $BD(\Omega)\setminus GBV(\Omega;\R^n)$ which has $e(u)=0$ almost everywhere 
and no jump, so that $Eu=E^cu$, see Theorem \ref{t:purecantor}.

 One of the main properties of $BV$ functions is that they are approximately continuous
 $\calH^{n-1}$-almost everywhere away from the jump points, and not only $\calL^n$-almost everywhere as
 any integrable function. It is still unknown if this property holds also for functions in $BD$.
 In Section~\ref{sbd-continuity} we show that $SBD^p$ functions, $p>1$, are approximately 
 continuous at $\calH^{n-1}$-almost every point away from the jump set, see Theorem~\ref{SuJu}.

 Finally, in view of the previous results, in Section~\ref{discussion} we discuss 
 the possibility that $SBD^p$ functions are actually of bounded variation.

\section{Caccioppoli-affine functions have (generalized) bounded variation}\label{sbd-caccioppoli}

In this section we show that piecewise affine functions induced by Caccioppoli partitions
have components in $GSBV$. For the theory of Caccioppoli partitions we refer to \cite[Section 4.4]{ambrosio}. 
In particular, for a given open set $\Omega\subset\R^n$ we consider a countable family $\mathscr{E}=(E_k)_k$ 
of sets $E_k\subseteq\Omega$ with finite perimeter in $\Omega$ satisfying 
\[
\calL^n\big(\Omega\setminus \cup_kE_k\big)=0,\quad
\calL^n(E_i\cap E_k)=0 \text{ for } i\neq k,\quad \sum_k\calH^{n-1}(\partial^\ast E_k\cap \Omega)<\infty,
\]
where $\partial^\ast E_k$ is the essential  boundary of $E_k$ (see \cite[Definition 3.60]{ambrosio}).
The \emph{set of interfaces} $J_{\mathscr{E}}$ of the Caccioppoli partition $\mathscr{E}$ is defined as 
the union of the sets $\partial^\ast E_k\cap \Omega$. It is known that
 \cite[Theorem 4.23]{ambrosio}. 
\[
\calH^{n-1}(J_{\mathscr{E}})=\frac12\sum_k\calH^{n-1}(\partial^\ast E_k\cap \Omega).
\]
We call a function $u:\Omega\to\R^m$ \emph{Caccioppoli-affine} if there exist matrices $\A_k\in\R^{m\times n}$ 
and vectors ${b}_k\in\R^m$ such that
\begin{equation}\label{e:caf}
u(x)=\sum_k\big(\A_kx+{b}_k\big)\chi_{E_k}(x)\,,
\end{equation}
where $(E_k)_k$ is a Caccioppoli partition of $\Omega$.

Functions of this type have already been studied in the literature. 
In particular, in \cite[Theorem A.1]{cha-gia-pon} it was proven that 
any function $u\in SBD(\Omega)$ with $e(u)=0$ $\calL^n$-a.e. on $\Omega$ and $\calH^{n-1}(J_u)<\infty$
is a \emph{piecewise rigid displacement}, which is defined as Caccioppoli-affine function
with $m=n$ and the matrices $A_k$ skew-symmetric.
Moreover, piecewise rigid displacements have been employed by Dal Maso \cite{gbd} to provide examples 
of fields in $GSBD\setminus SBD(\Omega)$. Elementary variations of such constructions 
prove that $SBD^p(\Omega)\setminus SBV^p(\Omega;\mathbb{R}^n)\neq\emptyset$ for every $p>1$. 
Even though this is probably a well-known fact we provide an explicit construction, 
since we have not been able to find any reference in literature. 
\begin{example}\label{example}

Set $\Omega:=B_1(0)\subset\R^n$ and for every $k\in \mathbb{N}$, choose $B_k:=B_{r_k}(x_k)\subset \Omega$,
so that the $B_k$ are pairwise disjoint, the centers $x_k$ converge 
to some point $x_\infty$, and the radii are $r_k:=2^{-k}$. 
We define 
\begin{equation*}
u(x):=\sum_k d_kA(x-x_k)\chi_{B_k}(x)\,,
\end{equation*}
where $A=(a_{ij})\in\R^{n\times n}$ with $a_{12}=-a_{21}=1$ and $a_{ij}=0$ otherwise, and 
$$
d_k:=\frac{2^{nk}}{k^2}\,.
$$ 
Then $u\in SBD(\Omega)\cap SBV(\Omega;\Rn)\cap L^\infty(\Omega;\Rn)$ with $e(u)=0$ $\mathcal{L}^n$-a.e. and 
$\Hn(J_u)<\infty$, but for every $q>1$ one has $\nabla u\notin L^q(\Omega;\R^{n\times n})$.  
\end{example}

Nevertheless Caccioppoli-affine displacements belong to $GSBV(\Omega;\R^m)$. We recall that $w\in GSBV(\Omega;\R^m)$ 
if $\phi(w)\in SBV_\mathrm{loc}(\Omega)$ for all $\phi\in C^1(\R^m)$ such that $\nabla\phi$ has compact support (cp. 
\cite[Definition~4.26]{ambrosio}). In the scalar case on can reduce simply to truncations (cp. \cite[Remark~4.27]{ambrosio})
\begin{theorem}\label{theocacciopp}
Let $\Omega\subseteq\R^n$ be a bounded Lipschitz set, 
$u:\Omega\to\R^m$ be Caccioppoli-affine. Then $u\in\big(GSBV(\Omega)\big)^m\subset GSBV(\Omega;\R^m)$ with 
\begin{equation}\label{e:deru}
\nabla u=\A_k\quad \text{$\calL^n$-a.e. on $E_k$, and }\quad \calH^{n-1}(J_u\setminus J_{\mathscr{E}})=0,
\end{equation}
(in the last formula we use the notation introduced in \eqref{e:caf}).

In particular, if $m=n$ and $u\in SBD(\Omega)$ with $e(u)=0$ $\calL^n$-a.e. on $\Omega$ 
and $\calH^{n-1}(J_u)<\infty$ then $u\in \big(GSBV(\Omega)\big)^m\subset GSBV(\Omega;\R^n)$.
\end{theorem}
\begin{proof} 
It suffices to prove the first assertion in the case $m=1$, assuming that $u$ is a scalar map of the form
$u=\sum_k(A_k\cdot x+b_k)\chi_{E_k}$, with $A_k\in\Rn$ and $b_k\in\R$.

For all $k\in\mathbb{N}$ consider $v_k:=A_k\cdot x+b_k$ and $u_k:=v_k\chi_{E_k}$ 
so that $u=\sum_ku_k$. Clearly, $u_k\in SBV(\Omega)$ with
\begin{equation}\label{e:dervj}
Du_k=A_k\,\calL^n\res E_k+v_k\,\nu_{E_k}\,\calH^{n-1}\res(\partial^\ast E_k\cap \Omega),
\end{equation}
with $\nu_{E_k}$ the generalised inner normal to $E_k$.
For $\Phi\in C^1(\overline{\Omega};\mathbb{R}^n)$ we use the divergence theorem to calculate
\[
\int_\Omega u_k\,\mathrm{div}\,\Phi\,dx+\int_\Omega\Phi\cdot d\,Du_k=
\int_{\partial\Omega\cap\partial^\ast E_k}\,v_k(\Phi\cdot\nu_{\partial \Omega})\,d\calH^{n-1},
\]
where $\nu_{\partial \Omega}$ is the outer normal to $\partial\Omega$. The specific choice $\Phi\equiv A_k$ 
and \eqref{e:dervj} yield
\[
|A_k|^2\,\calL^n(E_k)=
-\int_{\partial^\ast E_k\cap \Omega}v_k\,A_k\cdot\nu_{E_k}\,d\calH^{n-1}
+\int_{\partial^\ast E_k\cap \partial\Omega}v_k\,A_k\cdot\nu_{\partial\Omega}\,d\calH^{n-1},
\]
and, as $\nu_{E_k}=-\nu_{\partial\Omega}$ $\calH^{n-1}$-a.e. on
$\partial^\ast E_k\cap \partial\Omega$, we conclude that
\begin{multline}\label{e:gradukvk}
\int_{E_k}|\nabla u_k|\,dx=
-\int_{\partial^\ast E_k}v_k\,\frac{A_k}{|A_k|}\cdot\nu_{E_k}\,d\calH^{n-1}
\leq\int_{\partial^\ast E_k}|v_k|\,d\calH^{n-1}.
\end{multline}
In particular, setting $w_j:=\sum_{k=1}^ju_k$, it is clear that $w_j\in SBV(\Omega)$ and that
\[
Dw_j=\sum_{k=1}^jDu_k.
\]
Therefore, summing on $k\in\{1,\ldots,j\}$ inequality \eqref{e:gradukvk}, we deduce that
\begin{equation}\label{e:vartot}
|D w_j|(\Omega)\leq \sum_{k=1}^j|D u_k|(\Omega)
\leq 2\sum_{k=1}^j\int_{\partial^\ast E_k} |v_k| d\calH^{n-1}.
\end{equation}
Hence, assuming $u\in L^{\infty}(\Omega)$ we conclude for all $j\in\N$ that 
\begin{equation*}
|Dw_j|(\Omega)\leq 4\|u\|_{L^\infty(\Omega)}
\,\big(\calH^{n-1}(J_{\mathscr{E}})+\calH^{n-1}(\partial\Omega)\big).
\end{equation*}
In particular,  $(w_j)_j$ has equi-bounded $BV$ norm. Since
$w_j\to u$ in $L^1(\Omega)$ we conclude that $u\in BV(\Omega)$.
Actually, the sequence $(w_j)_k$ converges in $BV(\Omega)$ norm in view of \eqref{e:vartot}, 
and being $SBV(\Omega)$ a closed subspace of $BV(\Omega)$ we infer that $u\in SBV(\Omega)$. 

In the general case the conclusion $u\in GSBV(\Omega)$ follows by applying the argument above
to the truncated functions $w^M_j:=\sum_{k=1}^j\phi_M(u_k)$, $\phi_M(t):=t\wedge M \vee(-M)$, 
$M\in\N$, and using the chain rule formula to compute the distributional derivative. 
The identitities in \eqref{e:deru} are a consequence of \eqref{e:dervj}.

The second assertion follows immediately using the mentioned \cite[Theorem A.1]{cha-gia-pon}.
\end{proof}

\begin{remark}
The first assertion in Theorem~\ref{theocacciopp} is optimal: examples of Caccioppoli affine functions 
in $GSBV\setminus BV(\Omega)$, though not in $BD(\Omega)$, can be easily constructed. 
\end{remark}

\section{Pure-jump $BD$ functions not in $BV$}\label{sbd-purejump}

Contrary to the previous section, functions with vanishing symmetrized strain and jump set of infinite measure 
are not necessarily in the space $GSBV$.

\def\rank{\mathrm{rank}}
\begin{theorem}\label{t:purejump}
For any nonempty open set $\Omega\subseteq\R^n$ there is $u\in SBD(\Omega)\cap L^\infty(\Omega;\Rn)$
such that $e(u)=0$ $\calL^n$-a.e., $\Hn(J_u)=\infty$, and $\nabla u\notin L^1(\Omega;\Rn)$. 
In particular $u\not\in GBV(\Omega;\Rn)$.
\end{theorem}

Our construction is based on a suitable sequence of piecewise affine functions. More precisely we consider 
functions which are piecewise affine on polyhedra, but not necessarily continuous. 
Since there are finitely many pieces and each has a boundary of finite length, these functions all belong to $SBV$.
We recall that a convex polyhedron is a bounded set which is the intersection of finitely many half-spaces.
\begin{definition}
 For a convex polyhedron $\Omega\subset\R^n$ we let $PA(\Omega)$ be the set of functions $u:\Omega\to\R^n$ for which
 there is a decomposition of $\Omega$ into finitely many convex polyhedra such that $u$ is affine on each of them.
  \end{definition}

The next Lemma gives the basic construction step, which corresponds to the lamination used in 
\cite{ContiFaracoMaggi2005}.
  \begin{lemma}\label{lemmalaminatebd}
  Let $u\in PA(\Omega)$, $A,B\in\R^{n\times n}$
  with $\rank(A-B)=1$, $\lambda\in (0,1)$, $C:=\lambda A+(1-\lambda)B$,
  $\omega:=\{x\in\Omega: Du(x)=C\}$.
  For every $\eps>0$ there is $v\in PA(\Omega)$ such that
  $v=u$  on $\Omega\setminus\omega$,
  $Dv\in\{A,B\}$ $\calL^n$-a.e. in $\omega$,
  $\calL^n\big(\{x\in\omega: Dv(x)=A\}\big)=\lambda\,\calL^n(\omega)$, 
  and
  \bes{
	\|u-v\|_{L^\infty(\Omega,\Rn)}\leq \eps,\qquad
  \int_{J_v\cap\Omega} |[v]| d\calH^{n-1}\le \eps + 
   \int_{J_u\cap\Omega} |[u]| d\calH^{n-1}\,.
  }
\end{lemma}
\begin{proof}
 Let $\omega_1,\dots,\omega_k$ be the polyhedra where $Du=C$. 
 Given $\omega_j$, let $w_j\in\Lip(\R^n;\R^n)$ be a function with $Dw_j\in \{A,B\}$
 $\calL^n$-a.e. in $\omega_j$ and
$$\|w_j-u\|_{L^\infty(\omega_j,\Rn)}\le \eps\min\Big\{1,\frac{1}{4k\calH^{n-1}(\partial\omega_j)}\Big\}\,.$$
To construct $w_j$, let $a\in \R^n$ and $\nu\in S^{n-1}$ be such that $A=C+(1-\lambda)a\otimes \nu$, 
and set $w_j(x):=u(x)+N^{-1} a h_\lambda(N x\cdot \nu)$, where $h_\lambda\in \Lip(\R)$ is the $1$-periodic
function which obeys $h_\lambda(0)=0$, $h_\lambda'=1-\lambda$ on $(0,\lambda)$, $h_\lambda'=-\lambda$ on $(\lambda,1)$
and $N$ is sufficiently large
(see for example \cite[Lemma 4.3]{MuellerLectureNotes} or \cite[Section~2 and Lemma~3]{ContiFaracoMaggi2005} for details). 

By translation we can ensure that the volume fractions are the stated ones. Indeed, let $S_A:=\{x\in\Rn:\,Dw_j=A\}$ and $\nu$ as above. Then Fubini's 
theorem gives
\bes{
\int_0^1\Ln\big(\omega_j\cap (S_A-t\nu)\big)dt
=\int_{\Rn} \Big[\chi_{\omega_j}(y)\int_0^1 \chi_{S_A}(y+t\nu)dt\Big]dy=\lambda\,\Ln(\omega_j),
}
where the last equality follows from the fact that $\chi_{S_A}$ is $1/N$-periodic in the direction $\nu$
and $\chi_{S_A}=1$ on strips each of length $\lambda/N$. Choosing a suitable $t$ through the mean value theorem, the 
conclusion follows setting $v(x):=w_j(x+t\nu)$ on $\omega_j$.
\end{proof}

In our argument we shall iteratively apply Lemma \ref{lemmalaminatebd} above to increase the contribution of the skew-symmetric part
of the gradient without changing the contribution of the symmetric part. 
In order to be sure that no Cantor term in the
distributional derivative is created, we  stop the process after finitely many steps, and introduce
an additional iteration later. To treat the remainder zones where the piecewise affine function
has symmetric gradient, we approximate it with piecewise constant functions.

\begin{lemma}\label{lemmaaffinejump}
 Let $\omega$ be a convex polyhedron, and let $u:\omega\to\R^n$ be affine. For every $\eps>0$
 the following holds:
 \begin{itemize}
  \item[(i)] There is $v\in PA(\omega)$ 
 such that $\|u-v\|_{L^\infty(\omega,\Rn)}\le\eps$, $\nabla v=0$ $\calL^n$-a.e., and 
\be{\label{volume}|Dv|(\omega)\leq n|Du|(\omega).}
 \item[(ii)] There is  $v\in BV\cap C^0(\omega;\Rn)$ such that $\|u-v\|_{L^\infty(\omega,\Rn)}\le\eps$, $\nabla v=0$ $\calL^n$-a.e., and
$Dv=D^cv$ with $|Dv|(\omega)\leq n|Du|(\omega)$.
 \end{itemize}
\end{lemma}
\begin{proof} 
Let $u(x)=A\,x+b$ on $\omega$. For $\delta>0$ we set
    \begin{equation*}
     v^\delta(x) := \sum_i A\,e_i \left\lfloor \frac{x_i}{\delta}\right\rfloor \delta\,+b,
    \end{equation*}
		where $\left\lfloor \alpha\right\rfloor$ denotes the integer part of $\alpha\in\R$.
    It is easy to see that     
$    \|u-v^\delta\|_{L^\infty(\omega,\Rn)}\le \|A\|n\delta $ and
    \begin{equation*}
     \limsup_{\delta\to0} |Dv^\delta|(\omega) \leq \sum_i |A\,e_i| \, \calL^n(\omega) \le \sqrt n 
     \|A\|\,\calL^n(\omega)=\sqrt{n}|Du|(\omega)\,.
    \end{equation*}
Taking $\delta$ sufficiently small the proof is concluded.

To prove the second assertion we let $\Psi\in C^0([0,1];[0,1])$ be the usual Cantor staircase,
with $\Psi(0)=0$, $\Psi(1)=1$, $\Psi'=0$ $\calL^1$-almost evereywhere, and define
\[
 v^\delta(x):=
 \sum_iAe_i\,\delta\Big(\left\lfloor\frac{x_i}{\delta}\right\rfloor
 +\Psi\Big(\frac{x_i}{\delta}-\left\lfloor\frac{x_i}{\delta}\right\rfloor\Big)\Big)+b\,.
\]
\end{proof}

We are now ready to provide the main step in our argument.
\begin{lemma}\label{lemmaconstrcont}
 Let $\Omega$ be a convex polyhedron and let $M>1$. Then there is $u\in PA(\Omega)$ such that 
 $\|u\|_{L^\infty(\Omega,\Rn)}\leq c$, $e(u)=0$ $\calL^n$-a.e., $u=0$ in a neighborhood of $\partial\Omega$, 
 $|Eu|(\Omega)\le 1/M$, and $\|\nabla u\|_{L^1(\Omega,\R^{n\times n})}\ge M$. The constant $c$ depends only on the dimension $n$.
\end{lemma}
\begin{proof}
We define, for $k\in\N$, the matrices $A_k$, $B_k$, $C_k\in\R^{n\times n}$ by
 \begin{equation*}
  A_k:=\begin{pmatrix}
       0 & 2^k\\ 2^k & 0
      \end{pmatrix}\,, \hskip5mm
B_k:=\begin{pmatrix}
       0 & 2^k\\ -2^k & 0
      \end{pmatrix}\,, \hskip5mm
C_k:=\begin{pmatrix}
       0 & 2^k\\ 2^{k+1} & 0
      \end{pmatrix}\,,
 \end{equation*}
with the other entries vanishing if $n>2$.
 We construct inductively a sequence of functions $u_k\in PA(\Omega)$
 and sets $\omega_k\subset\Omega$, with $\Omega_k$ a finite union of convex polyhedra, such that
 $\Omega_k:=\{x\in\Omega:\,Du_k=A_k\}$ for every $k$ and $e(u_k)=0$ 
 $\Ln$-a.e. on $\Omega\setminus\Omega_k$.
 
We start with a convex polyhedron $\Omega_0\subset\subset\Omega$ with $\calL^n(\Omega_0)\le 1$ 
and the function $u_0:= \chi_{\Omega_0}A_0x$.

In order to construct $(u_{k+1}, \Omega_{k+1})$ from $(u_k,\Omega_k)$, we observe that 
\begin{equation*}
 A_k = \frac13 B_k + \frac23 C_k\quad \text{and}\quad \rank (B_k-C_k)=1\,.
\end{equation*}
We define $\hat u_k$ by  Lemma \ref{lemmalaminatebd}, with $\eps=2^{-k}$, $\omega=\Omega_k$,
and $C=A_k$.
We set $\hat\Omega_k:=\{x\in\Omega:\,D\hat u_k(x)=C_k\}$ and note that $\hat\Omega_k\subset\Omega_k$ and 
$\calL^n(\hat\Omega_k)=\frac23\calL^n(\Omega_k)$. Since 
\begin{equation*}
 C_k = \frac34 A_{k+1} + \frac14 (-B_{k+1})\quad \text{and}\quad \rank (A_{k+1}+B_{k+1})=1\,,
\end{equation*}
we can apply  Lemma \ref{lemmalaminatebd} again to $\hat u_k$, with the same $\eps$, $\omega=\hat\Omega_k$, 
and $C=C_k$, to obtain
$u_{k+1}\in PA(\Omega)$ such that,  with 
$\Omega_{k+1}:=\{x\in\Omega:\,Du_{k+1}(x)=A_{k+1}\}\subset\hat\Omega_k$, 
it holds $e(u_{k+1})=0$ on $\Omega\setminus\Omega_{k+1}$, $\calL^n(\Omega_{k+1})=\frac12\calL^n(\Omega_k)=
2^{-(k+1)}\calL^n(\Omega_0)$,
\begin{equation}\label{e:saltouk}
 \int_{\Omega\cap J_{u_{k+1}}} |[u_{k+1}]| d\calH^{n-1} \le 2\cdot 2^{-k} + 
 \int_{\Omega\cap J_{u_{k}}} |[u_{k}]| d\calH^{n-1}, 
\end{equation}
and
\begin{align}\label{e:graduk}
 \int_{\Omega\setminus\Omega_{k+1}} |\nabla u_{k+1}| dx&=
\int_{\Omega\setminus\Omega_{k}} |\nabla u_k| dx
+ \frac13 \calL^n(\Omega_k) \, |B_k| + \frac16\calL^n(\Omega_k) \, |B_{k+1}|\notag\\
&=\int_{\Omega\setminus\Omega_{k}} |\nabla u_k| dx
+ \frac23 \sqrt 2\, \calL^n(\Omega_0).
\end{align}
Therefore, in view of \eqref{e:saltouk} and \eqref{e:graduk}, for all $k$ we conclude
\be{\label{Euk}|Eu_k|(\Omega)\le |A_k|\, \calL^n(\Omega_k)+ \sum_k 2\cdot 2^{-k} +\int_{\partial\Omega_0}|A_0x|d\calH^{n-1}\le c} 
and $\|\nabla u_k\|_{L^1(\Omega,\R^{n\times n})}\to\infty$. 

Let $\{\Omega_k^i\}_{i=1}^N$ be the set of polyhedra which composes $\Omega_k$, and let $v_k^i$ be the 
function provided by Lemma~\ref{lemmaaffinejump}(i) applied to $u_k$ on $\Omega_k^i$ with 
$\eps_i:=\min\Big\{1,1/\big(N\,\calH^{n-1}(\partial\Omega_k^i)\big)\Big\}$. Recall that 
$ \|u_k-v_k^i\|_{L^{\infty}(\Omega^i_k,\Rn)}\leq\eps_i$.

We define $w_k:=u_k\chi_{\Omega\setminus\Omega_k}+\sum_{i=1}^Nv_k^i\chi_{\Omega_k^i}$. 
Note that $w_k\in PA(\Omega)$, $\nabla w_k=\nabla u_k\chi_{\Omega\setminus\Omega_k}$, $e(w_k)=0$ $\calL^{n-1}$-a.e. 
on $\Omega$, and 
\ba{\label{saltovk} 
& \displaystyle\int_{J_{w_k}\cap\overline{\Omega}_k}|[w_k]|d\calH^{n-1}=
\sum_{i=1}^N\int_{J_{v_k^i}\cap\Omega_k^i}|[v_k^i]|d\calH^{n-1}+\sum_{i=1}^N\int_{\partial\Omega_k^i}|[w_k]|d\calH^{n-1}\nonumber\\
&  \displaystyle\stackrel{\eqref{volume}}{\leq} n|Du_k|(\Omega_k)+ 
\sum_{i=1}^N\|u_k-v_k^i\|_{L^{\infty}(\Omega^i_k,\Rn)}\calH^{n-1}(\partial\Omega_k^i)
+\int_{\partial\Omega_k}|[u_k]|d\calH^{n-1}\nonumber\\
&\leq
\displaystyle c\,\calL^n(\Omega_0)+1+\int_{J_{u_k}\cap\Omega}|[u_k]|d\calH^{n-1}\leq c,
}
for some constant $c$ independent from $k$ thanks to \eqref{e:saltouk}.

In conclusion by \eqref{Euk} and \eqref{saltovk}
$$
|Ew_k|(\Omega)\leq
|Eu_k|(\Omega\setminus\overline\Omega_k)+\int_{J_{w_k}\cap\overline\Omega_k}|[w_k]|d\calH^{n-1}\leq c,
$$
for a constant $c$ independent from $k$ (and $M$). 
The function
$w_k/(cM)$ has the stated properties for $k$ sufficiently large.
\end{proof}

We are now ready to prove the main result of the section.
\begin{proof}[Proof of Theorem~\ref{t:purejump}]
Let $\{Q_k\}_{k\in\N}$ be a family of  countably many disjoint cubes contained in $\Omega$. For each of them let
$u_k$ be the function constructed in Lemma~\ref{lemmaconstrcont} above, using $M:=2^k$. We set $u:=u_k$ in $Q_k$,  
$u:=0$ on the rest.  
Note that the sequence $(\sum_{k=1}^ju_k\chi_{Q_k})_j$ converges in the $BD$ norm to $u$ as $|Eu_k|(Q_k)\le 2^{-k}$. 
Therefore, $u\in SBD(\Omega)$, $e(u)=0$ $\calL^n$-a.e.,  $|Eu|(\Omega)=\sum_k |Eu|(Q_k)\le 2$ and
$$\|\nabla u\|_{L^1(\Omega,\R^{n{\times}n})}=\sum_k \|\nabla u_k\|_{L^1(Q_k,\R^{n{\times}n})}\ge \sum_k 2^k=\infty.$$
Finally, $u\in L^\infty(\Omega;\R^n)$ by Lemma~\ref{lemmaconstrcont}, thereby $u\notin GBV(\Omega;\R^n)$.
\end{proof}

A slight modification of the previous construction provides a function $u$ 
in $BD\setminus GBV$ for which $Eu=E^cu$.

\begin{theorem}\label{t:purecantor}
For any nonempty open set $\Omega\subseteq\R^n$ there is $u\in BD(\Omega)\cap L^\infty(\Omega;\Rn)$
such that $Eu=E^cu$ and $\nabla u\notin L^1(\Omega;\Rn{\times}\Rn)$. In particular $u\not\in GBV(\Omega;\Rn)$.
\end{theorem}
\begin{proof}
 The proof is similar to the one of Theorem~\ref{t:purejump}, therefore we only highlight
 the significant changes in the construction. For notational simplicity we focus on the two-dimensional
 situation with $\Omega=(0,1)^2$.
 
 We introduce first  cut-off functions whose gradient only has a Cantor part. 
 Given $\ell>0$ and $\delta\in(0,\ell/2)$, we define 
 $\psi_{\ell,\delta}:[-\ell,\ell]\to[0,1]$ to be
 \[
\psi_{\ell,\delta}(t):=  
\begin{cases}
1 & t\in[-\ell+\delta, \ell-\delta]\,, \cr
\Psi\Big(\frac{\ell-|t|}{\delta}\Big) & \text{otherwise},
\end{cases}
 \]
 where $\Psi:[0,1]\to[0,1]$ is the  Cantor staircase as in the proof of Lemma \ref{lemmaaffinejump}.
For a rectangle $R=[-a,a]{\times}[-b,b]$  we set
\[
\psi_{R,\delta}(x_1,x_2):=\psi_{a,\delta}(x_1)\psi_{b,\delta}(x_2).   
\]
Note that $\psi_{R,\delta}\in BV\cap C^0(R)$ with 
$$
D\psi_{R,\delta}=D^c\psi_{R,\delta} \quad \text{and}\quad 
|D^c\psi_{R,\delta}|(R)\leq \calH^1(\partial R).
$$ 
Moreover, $\psi_{R,\delta}=1$ on 
$[-a+\delta,a-\delta]\times[-b+\delta,b-\delta]$, $\psi_{R,\delta}|_{\partial R}=0$ 
so that its extension to $0$ on $R^c$ provides a function $BV\cap C^0(\R^2)$ without
altering the total variation. 
 
We fix $\gamma>0$.
We perform the same iterative construction as in Lemma \ref{lemmaconstrcont}.
We start with $\Omega_0=(0,1)^2$ and $u_0=A_0x$.
At step $k$ we use Lemma~\ref{lemmalaminatebd} to construct from $u_{k-1}$ and $\Omega_{k-1}$ 
the functions $\hat u_k$ and $u_k$, and the sets $\Omega_k$ and $\hat\Omega_k$. However,
at the $k$-th step we apply Lemma~\ref{lemmalaminatebd} with $\eps_k=2^{-k}\gamma \min\{1,1/\calH^1(\partial\hat\Omega_k^i)\}$, where  $\{\Omega_k^i\}_{i=1}^{N_k}$ and $\{\hat\Omega_k^i\}_{i=1}^{M_k}$ are the  the polyhedra composing
 $\Omega_k$ and $\hat\Omega_k$, respectively. 
This concludes the construction of $u_k$ and $\Omega_k$.
We remark that for each $k$ the $(\Omega_k^i)_i$
are disjoint congruent rectangles, and analogously the  $(\hat\Omega_k^i)_i$.  
 
We now construct a sequence of modified functions $U_k$, where each jump is replaced by a continuous Cantor staircase. To do this,
at each $k$ 
we consider the Cantor-type cut-off functions 
$\psi_{k}:=\sum_i\psi_{\Omega_k^i,\delta_k}$ and $\hat\psi_{k}:=\sum_i\hat\psi_{\Omega_k^i,\delta_k}$, where 
$\delta_k>0$ will be suitably chosen below. We set $U_0:=u_0$ and 
\be{\label{e:Uk}\begin{cases}
\hat U_k:=(1-\psi_{k}) U_{k}+ \psi_{k}\hat u_k& \text{ for $k\geq 0$}\,,\\
U_k:=(1-\hat\psi_{k-1})\hat U_{k-1}+ \hat\psi_{k-1} u_k& \text{ for $k\geq 1$}\,.
\end{cases}
}
Finally, the truncation step at the end is different. We 
let $v_k^i$ be the function provided by Lemma \ref{lemmaaffinejump}(ii) applied to $u_k$ on $\Omega_k^i$ 
with $\eps_k=\min\Big\{1,1/\big(N_k\,\calH^1(\partial\Omega_k^i)\big)\Big\}$. Recall that 
$ \|u_k-v_k^i\|_{L^{\infty}(\Omega^i_k,\Rn)}\leq\eps_k$ and $Dv_k^i=D^cv_k^i$. We extend $v_k^i$ to $0$
on $\Omega\setminus\Omega_k^i$ and set $v_k:=\sum_iv^i_k$.

Let then 
\[
w_k:=(1-\psi_k)\,U_k+\psi_kv_k\,.
\]
By \cite[Example~3.97]{ambrosio}, $w_k\in BV\cap C^0(\Omega_0;\R^2)$, furthermore by construction 
$Dw_k=\nabla w_k\calL^2\res\Omega_0 +D^cw_k$ with $\nabla w_k$ skew-symmetric $\calL^2$-a.e. in
$\Omega_0$, therefore $Ew_k=E^cw_k$. More precisely, the chain-rule formula yields
\begin{equation}\label{dedcwk}
D^cw_k=(1-\psi_k)D^cU_k+\psi_kD^cv_k+(v_k-U_k)\otimes D^c\psi_k.
 \end{equation}
In what follows we shall estimate separately the total variations of the three terms in (\ref{dedcwk}). To this aim
we first note that in view of \eqref{e:Uk} we get 
\bes{
\|u_k-U_k\|_{L^\infty(\Omega,\R^n)}\leq 2\,\sum_{j=0}^k\eps_j\leq \frac{4\gamma}{\calH^1(\partial\hat\Omega_k^i)}
=:\sigma_k.
}
To bound the first term in (\ref{dedcwk}) we notice that for some positive constant $c$ independent from $k$ we have
\[
|D^cU_k|(\Omega_i^k)\stackrel{\eqref{e:Uk}}{\leq}|D^cU_{k-1}|(\Omega_i^k)
+c\,\sigma_{k-1}\calH^{1}(\partial\Omega^i_{k-1})+c\,\sigma_{k}\calH^{1}(\partial\hat\Omega^i_{k}).
\]
By summing over $i$ we get
\bes{
|(1-\psi_k)D^cU_k|(\Omega_0)\leq c\, \gamma\,N_k.
}
From Lemma \ref{lemmaaffinejump}(ii) we infer that
\bes{
|\psi_kD^cv_k|(\Omega_0)\leq |D^cv_k|(\Omega_k)\leq 2\, |Du_k|(\Omega_k)=2\,|Eu_k|(\Omega_k)
\stackrel{\eqref{Euk}}{\leq} c.
}
We turn to the last term of (\ref{dedcwk}). By the triangle inequality
\[
|(v_k-U_k)\otimes D^c\psi_k|(\Omega_k^i)\leq\|v_k-U_k\|_{L^\infty(\Omega_i^k,\R^n)}|D^c\psi_k|(\Omega_i^k)
\leq(\eps_k+\sigma_k)\calH^1(\partial\Omega_k^i).
\]
Therefore
\bes{
|(v_k-U_k)\otimes D^c\psi_k|(\Omega_0)=|(v_k-U_k)\otimes D^c\psi_k|(\Omega_k)\leq c(1+\gamma\,N_k).
}
Collecting terms, we conclude that
\be{\label{e:I3}
|D^cw_k|(\Omega_0)\leq c(1+\gamma\,N_k).
}
Finally, by arguing as in \eqref{e:graduk} taking into account the definitions in \eqref{e:Uk} we infer that
\bes{
\int_{\Omega\setminus\Omega_k}&&\hskip-0,4cm|\nabla w_k|dx=
\int_{\Omega\setminus\Omega_k}|\nabla U_k|dx\geq
\int_{\Omega\setminus\Omega_{k-1}}|\nabla U_{k-1}|dx\nonumber\\
&&+\|B_k\|\calL^2(\{\hat\psi_{k-1}=1\}\setminus\Omega_k)+|B_{k-1}|\calL^2(\{\psi_{k-1}=1\}\setminus\hat\Omega_{k-1})\nonumber\\
&&\geq\int_{\Omega\setminus\Omega_{k-1}}|\nabla U_{k-1}|dx
+\frac 17\big(|B_k|+|B_{k-1}|\big)\calL^2(\Omega_{k-1}),
}
by suitably choosing $\delta_k$ in the definition of $\psi_j$ and $\hat\psi_j$, $0\leq j\leq k$.
Here $B_k$ is the matrix defined in the proof of Lemma \ref{lemmaconstrcont}.
In particular, we conclude that  $\|\nabla w_k\|_{L^1(\Omega,\R^{2\times 2})}\to\infty$.

Choosing $k_*$ large enough to have $\|\nabla w_{k_*}\|_{L^1(\Omega_0,\R^{2\times 2})}> c\,M^2$,
the function $z:=w_{k_*}/(c\,M)$ satisfies $\|\nabla z\|_{L^1(\Omega_0,\R^{2\times 2})}> M$,
and choosing $\gamma=1/N_{k_*}$ the bound \eqref{e:I3} gives
\[
|E^cz|(\Omega_0)\leq \frac 1M\,.
\]
To conclude we remark that by the definitions in \eqref{e:Uk} we have $z|_{\partial\Omega_0}=u_0|_{\partial\Omega_0}$.
Therefore repeating the same construction on a countable family of disjoint closed cubes 
and choosing $M=2^{-j}$ we obtain a function with the properties given in the statement.
\end{proof}

\section{Continuity away from fractures}\label{sbd-continuity}

Let $u\in BD(\Omega)$, with $\Omega$ open. We say that $x\in\Omega$ is a point
of approximate continuity of $u$ if there is $\tilde u(x)$ such that
\begin{equation*}
 \lim_{r\to0} \frac{1}{\Ln(B_r)} \int_{B_r(x)} |u(y)-\tilde u(x)| dy = 0\,.
\end{equation*}
We denote by $S_u$ the set of points $x\in\Omega$ which are not points of approximate continuity. 
Since $u\in L^1(\Omega;\Rn)$, one immediately has $\mathcal L^n(S_u)=0$.

We say that $x\in\Omega$ is a jump point of $u$ if there are two vectors $u_+\ne u_-\in\R^n$
and a normal $\nu\in S^{n-1}$ such that
\begin{equation*}
 \lim_{r\to0} \frac{1}{\Ln(B_r)} \int_{B_r^+(x) } |u(y)-u_+| dy = 
 \lim_{r\to0} \frac{1}{\Ln(B_r)} \int_{B_r^-(x)\ } |u(y)-u_-| dy = 0\,,
\end{equation*}
where $B^\pm_r(x)=B_r(x)\cap \{\pm (y-x)\cdot \nu >0\}$.
Obviously $J_u\subset S_u$. If $u\in BV$ it is well known that $\calH^{n-1}(S_u\setminus J_u)=0$, 
see, for example, \cite{ambrosio}.

We further define $\Theta_u$ as the set of points $x\in\Omega$ such that
\begin{equation*}
 \limsup_{r\to0}\frac{|Eu|(B_r(x))}{r^{n-1}}>0\,.
\end{equation*}
The set $\Theta_u$ is  $(n-1)$-rectifiable, $J_u\subset \Theta_u$, and 
$\calH^{n-1}(\Theta_u\setminus J_u)=0$, see \cite[Prop. 3.5]{amb-cos-dal}.

The main result of the section is the following theorem concerning the size of the set $S_u\setminus J_u$
when the function $u$ belongs to $SBD^p(\Omega)$.
\begin{theorem}\label{SuJu}
If $u\in SBD^p(\Omega)$ for some $p>1$, with $\Omega\subset\R^n$ open, then $\calH^{n-1}(S_u\setminus J_u)=0$.
\end{theorem}
The proof uses Lemma \ref{lemmapoincholes} and  Lemma \ref{lemmaFfLp}  below. The first 
is a consequence of the Korn-Poincar\'e inequality for $SBD^p$ functions proven in 
\cite{ChambolleContiFrancfort} and
states that a function
$u\in SBD$ with a small jump set can be approximated by an affine function away from a special set,
with an error depending only on $e(u)$.


\begin{lemma}\label{lemmapoincholes}
 There is $\eta>0$, depending only on $n$, such that the following holds.
 For any $r>0$ and $u\in SBD(B_r)$ such that $\calH^{n-1}(J_u)<\eta r^{n-1}$ there are a set 
 $\omega\subset B_r$ and an affine function $\varphi:\R^n\to\R^n$ such that 
 \begin{equation*}
  \calL^n(\omega) \le \frac1{2^{n+3}} \calL^n(B_r)
 \end{equation*}
and
 \begin{equation*}
  \int_{B_r\setminus \omega} |u-\varphi| dx\le c\,  r\, \int_{B_r} |e(u)|dx\,.
 \end{equation*}
\end{lemma}
\begin{proof}
 This follows from Theorem 1 of \cite{ChambolleContiFrancfort}.
\end{proof}
 The next lemma is used in the proof of Lemma \ref{lemmaFfLp} below. It shows that it is possible to control the
$L^\infty$-norm of an affine function through its $L^1$-norm out of a special set. 
\begin{lemma}\label{affine}
Let $\omega\subset B:=B_r(y)$ satisfy
\be{\label{omega}\calL^n(\omega)\leq\frac14 \calL^n(B)}
and let $\varphi:\R^n\to\R^n$ be an affine function. Then
\be{\label{affineest}\calL^n(B)\|\varphi\|_{L^\infty(B,\Rn)}\leq c\|\varphi\|_{L^1(B\setminus\omega,\Rn)},}
where the constant $c$ depends only on the dimension $n$.
\end{lemma}
\begin{proof}
By a scaling argument it is sufficient to prove the statement when $B=B_1$. First note that for sufficiently small $\delta>0$
one has for every $i=1,\dots,n$
$$\calL^n(B\cap (B+\delta e_i))\geq \frac34 \calL^n(B).$$
Hence the set 
$E_i:=(B\setminus\omega)\cap ((B\setminus\omega)+\delta e_i)$
satisfies 
\be{\label{Ei}\calL^n(E_i)\geq\frac34\calL^n(B)-2\calL^n(\omega)\geq \frac14 \calL^n(B).}
Estimating the translation for an affine map $\varphi(x)=Ax+b$ one gets
$$\delta |Ae_i|\calL^n(E_i)= \int_{E_i}|\varphi(x)-\varphi(x-\delta e_i)|dx\leq 2\int_{B\setminus\omega}|\varphi|dx$$
and therefore by (\ref{Ei})
\be{\label{A}|A|\leq \frac{8\sqrt{n}}{\delta\calL^n(B)}\|\varphi\|_{L^1(B\setminus\omega,\Rn)}.}
The constant $b$ can be estimated using (\ref{A}) and (\ref{omega})
\ba{\calL^n(B\setminus\omega)b&\leq&\|\varphi-\varphi(0)\|_{L^1(B\setminus\omega,\Rn)}+\|\varphi\|_{L^1(B\setminus\omega,\Rn)}\nonumber\\
&\leq&\frac{\calL^n(B\setminus\omega)}{\calL^n(B)}\Big(\frac{8\sqrt{n}}{\delta}+\frac43\Big)\|\varphi\|_{L^1(B\setminus\omega,\Rn)}.\label{b}}
Formula (\ref{affineest}) follows from (\ref{A}) and (\ref{b}).
\end{proof}

The following lemma refines the classical estimates for the $(n-1)$-dimensional density in a particular case.

\begin{lemma}\label{lemmaFfLp}
 Let $f\in L^p(\R^n)$, for some $p>1$, and
 \begin{equation*}
  F := \left\{x\in \R^n: \sum_{k\in\N} \frac{1}{r_k^{n-1}} \int_{B_{r_k}(x)} |f|\, dy =\infty\right\}
 \end{equation*}
 where $r_k=2^{-k}$.
Then $\calH^{n-1}(F)=0$. 
If additionally $p>n$ then $F=\emptyset$.
 \end{lemma}
\begin{remark}  To see that $p>1$ is required one can consider 
 $f(x):=\chi_{B_{1/2}}(x)/(x_n \ln^2(1/x_n))$.
\end{remark}
\begin{proof}
We start from the second statement.
From $\|f\|_{L^1(B_r)}\le r^{n/p'} \|f\|_{L^p(\R^n)}$, 
where $1/p'=1-1/p$,
one deduces, for any $x\in\R^n$, 
\begin{equation*}
  \sum_k \frac{1}{r_k^{n-1}} \int_{B_{r_k}(x)} |f| dy 
  \le \sum_k \frac{r_k^{n/p'}}{r_k^{n-1}} \|f\|_{L^p(\R^n)}
  =  \|f\|_{L^p(\R^n)}\sum_k 2^{k(n-p)/p}  \,.
\end{equation*}
If $p>n$ the last series converges and therefore $x\not\in F$.

\def\rj{r_{(j)}}
\def\rj{\rho_j}
We now turn to the first statement and assume $p\in(1,n]$. By Lebesgue's differentiation theorem one 
has $\Ln(F)=0$.

Fix $\delta>0$. For every $x\in F$ there is $k(x)\in\N$ such that $r_{k(x)}<\delta$ and
\begin{equation}\label{frho}
  \int_{B_{r_{k(x)}}(x)} |f| dy\ge \frac{r_{k(x)}^{n-1}}{k(x)^2},
\end{equation}
otherwise the series would converge.
By Vitali's $5r$-covering lemma, there is a family $G=\{{B_{\rj}(x_j)}: j\in\N\}$ of disjoint balls,
with $\rj=r_{k(x_j)}$ and $k_j=k(x_j)$, 
such that $F\subset \bigcup_j B_{5\rj}(x_j)$. Then
by the definition of the Hausdorff measure 
\begin{equation}\label{Hdelta}
 \calH^{n-1}_{10\delta}(F)\le c \sum_j (5\rj)^{n-1}. 
\end{equation}
By \eqref{frho} and by H\"older's inequality, we obtain
\begin{equation*}
 \rj^{n-1}\le 
 k_j^2 \int_{B_{\rho_j}(x_j)} |f| dy
 \leq c\|f\|_{L^p(B_{\rho_j}(x_j))}k_j^2\rj^{n/p'}\,.
 \end{equation*}
For all $M\in\N$, using H\"older's inequality and the fact that 
the balls are disjoint one obtains
\begin{alignat*}1
S_M:=\sum_{j=0}^M  \rj^{n-1} \le &
 c\,\|f\|_{L^p(\R^n)}\left( \sum_{j=0}^M  k_j^{2p'}\rj^n \right)^{1/p'} \,.
\end{alignat*}
For any $p\in(1,\infty)$ there is $\delta_p$ such that
$r^{1/2}|\log_2 r|^{2p'}\le 1$ for all $r\in (0,\delta_p)$. Therefore for all $\delta\le \delta_p$ 
one obtains
\begin{alignat}1\label{sumdelta}
S_M\le c\, \|f\|_{L^p(\R^n)}\,\delta^{1/(2p')} S_M^{1/p'}.
\end{alignat}
By collecting \eqref{Hdelta} and \eqref{sumdelta} we conclude
\begin{equation*}
 \calH^{n-1}_{10\delta}(F)\le c\,\|f\|_{L^p(\R^n)}^p\,\delta^{p/(2p')} ,
\end{equation*}
therefore with  $\delta\to0$ we obtain $\calH^{n-1}(F)=0$.
\end{proof}

\begin{proof}[Proof of Theorem \ref{SuJu}]
Let $x\in\Omega\setminus\Theta_u$, by \cite[Corollary 6.7]{amb-cos-dal}
for any $r>0$ there is $a_{x,r}\in\R^n$ such that
\begin{equation}\label{eqpoincareconasd}
  \limsup_{r\to0} \frac{1}{r^n} \int_{B_r(x)} |u(y)-a_{x,r}| dy=0\,.
\end{equation}
In the rest of the proof we will show that $\lim_{r\to0} a_{x,r}$ exists for $\calH^{n-1}$-almost every $x\in\Omega\setminus \Theta_u$.
In view of \eqref{eqpoincareconasd} it is actually sufficient to prove the claim for the subsequence $r_k=2^{-k}$.

Let $\eta$ be as in Lemma \ref{lemmapoincholes} and let
 \begin{equation*}
\Omega_k:=\{x\in \Omega: B_{r_k}(x)\subset\Omega \text{ and } \calH^{n-1}(J_u\cap B_{r_k}(x))\le \eta r_k^{n-1}\}\,.      
     \end{equation*}

We also introduce the set		
\begin{equation*}
  F := \left\{x\in \R^n: \sum_{k\in\N} \frac{1}{r_k^{n-1}} \int_{B_{r_k}(x)} |e(u)|\, dy =\infty\right\}
 \end{equation*}
and we fix $k\in\N$ and $x\in \bigcap_{h\ge k} \Omega_h\setminus (\Theta_u\cup F)$. 
By  Lemma \ref{lemmapoincholes} for every $h\ge k$ there are an affine map $\varphi_h:\R^n\to\R^n$ and a set $\omega_h$ such that
 $\calL^n(\omega_h)<2^{-n-3} \calL^n(B_{r_h})$ and
 \begin{equation}\label{equAhomeag}
  \int_{B_{r_h}(x)\setminus \omega_h} |u-\varphi_h|dy \le c r_h \int_{B_{r_h}(x)} |e(u)|dy\,.
 \end{equation}
 Lemma \ref{affine} yields 
 \begin{alignat}1 \nonumber
  |\varphi_h(x)-a_{x,r_h}|\le &\frac{c}{r_h^n} 
   \int_{B_{r_h}(x)\setminus \omega_h} |a_{x,r_h}-\varphi_h|dy
  \\ \le& \frac{c}{r_h^n} 
   \int_{B_{r_h}(x)\setminus \omega_h} |a_{x,r_h}-u|dy
+ \frac{c}{r_h^n} 
   \int_{B_{r_h}(x)\setminus \omega_h} |u-\varphi_h|dy\,.
   \label{eqah0axrh}
   \end{alignat}
At the same time, using the triangle inequality, (\ref{equAhomeag}) implies
 \begin{equation*}
  \int_{B_{r_{h+1}}(x)\setminus (\omega_h\cup \omega_{h+1})} |\varphi_h-\varphi_{h+1}|dy \le c r_h \int_{B_{r_h}(x)} |e(u)|dy\,.
 \end{equation*}
Lemma \ref{affine} now implies
 \begin{equation*}
 \sum_{h}\|\varphi_h-\varphi_{h+1}\|_{L^\infty(B_{r_h}(x),\Rn)} \le c \sum_h\frac{1}{r_h^{n-1}} \int_{B_{r_h}(x)} |e(u)| dy\,,
 \end{equation*}
since $\calL^n(B_{r_{h+1}}\setminus( \omega_h\cup\omega_{h+1}))\ge
(1-2^{-3}-2^{-n-3})\calL^n(B_{r_{h+1}})\ge \frac34 \calL^n(B_{r_{h+1}})$ and the maps $\varphi_h$ are affine.
Recalling that $x\not\in F$,  we have that
 the series $\sum_h|\varphi_h-\varphi_{h+1}|(x)$ converges and we can define $\tilde u(x):=\lim_{h\to\infty} \varphi_h(x)\in\R^n$.
 Then
\begin{alignat*}1
&\frac{1}{r_h^n}\int_{B_{r_h}(x)} |u(y)-\tilde u(x)| dy 
 \\&
 \le
\frac{1}{r_h^n}\int_{B_{r_h}(x)} |u(y)-a_{x,r_h}| dy 
+  c|a_{x,r_h}-\varphi_h(x)|+  c|\varphi_h(x)-\tilde u(x)|\\
& \le 
\frac{1}{r_h^n}\int_{B_{r_h}(x)} |u(y)-a_{x,r_h}| dy+  \frac{c}{r_h^{n-1}} \int_{B_{r_h}(x)} |e(u)| dy+c|\varphi_h(x)-\tilde u(x)|,
\end{alignat*}
where we have used \eqref{eqah0axrh}, \eqref{equAhomeag}.
 The first term converges to zero by (\ref{eqpoincareconasd}), the second one by the fact that $x\not\in F$, and the third one by the definition of $\tilde u$.
 Therefore $x\in \Omega\setminus S_u$. We conclude that
 \begin{equation*}
 S_u \subset \Theta_u\cup F \cup \left( \Omega \setminus 
  \bigcup_k\bigcap_{h\ge k} \Omega_h\right)\,.
 \end{equation*}
We finally show that 
\begin{equation}\label{OmegaJu}
\Omega\setminus \bigcup_k\bigcap_{h\ge k} \Omega_h \subset J_u \cup N 
\end{equation}
for some $\calH^{n-1}$-null set $N$,
from which we can conclude $S_u\subset \Theta_u \cup F\cup N$.
Since $\calH^{n-1}(J_u)<\infty$ 
by \cite[Th. 2.56]{ambrosio} there is a set $N$ with $\calH^{n-1}(N)=0$ such that
\begin{equation*}
 \limsup_{r\to0} \frac{\calH^{n-1}(J_u\cap B_r(x))}{r^{n-1}}=0 \text{ for all } x\in \Omega\setminus (J_u\cup N)\,.
\end{equation*}
In particular, for any such $x$ there is a $k$ such that $x\in \Omega_h$ for any $h\ge k$, so that formula \eqref{OmegaJu} holds.
From $e(u)\in L^p(\Omega;\R^{n\times n})$, Lemma \ref{lemmaFfLp} guarantees that $\calH^{n-1}(F)=0$, 
and the proof is concluded.
\end{proof}

\section{Discussion}\label{discussion}
As mentioned in the introduction, 
in the variational formulation of some mathematical problems in fracture mechanics, 
 the natural analytical framework is given by the set  
$SBD^p(\Omega)$ defined in (\ref{eqefsbdp}).
Due to the regularity generated by the higher integrability of the strain and  the finite $\Hn$-measure of 
the jump set, it has been conjectured that $SBD^p(\Omega)$ functions might have in fact bounded variation. 
In this case one would have in particular $SBD^p(\Omega)\subset SBV(\Omega;\Rn)$, since  Alberti's rank one theorem \cite{Alb93,AlbCsoPre05} implies $|D^cu|(\Omega)\leq\sqrt{2}|E^cu|(\Omega)$ for each $u\in BV(\Omega;\Rn)$. 
Moreover one could reasonably expect that the inclusion above is related to the validity of a Korn's type inequality involving 
the $L^1$-norm of $\nabla u$, the  $L^p$-norm of $e(u)$, and $\Hn(J_u)$. 
A Korn's type inequality for $\nabla u$ alone is, however, 
not enough to guarantee that a field $u$ in $SBD^p(\Omega)$ has bounded variation, since an equivalent of 
Alberti's rank one theorem in $BD$ has not been proved (nor disproved).

In light of the results described in the previous sections, we are now in the position to discuss the possible optimality of the immersion above. First of all let us recall that a generic function $u\in SBD(\Omega)$ has not necessarily bounded variation, even in the case in which $u$ has no jumps. This follows from the counterexample to Korn's inequality in $W^{1,1}$ built in \cite{ornstein}, see \cite[Example 7.7]{amb-cos-dal}. 
A similar idea is the starting point
of the construction presented in Section~\ref{sbd-purejump}, developed 
starting from \cite[Theorem 1]{ContiFaracoMaggi2005}. Indeed,
we have shown that, if no bound on the size of the jump set is given, not even the condition 
$e(u)=0$ $\mathcal{L}^n$-a.e. guarantees that a function $u\in SBD(\Omega)$ has bounded variation.
Analogously, we have constructed a function $u\in BD(\Omega)\setminus BV(\Omega;\R^n)$
with $Eu=E^cu$. 

On the contrary, in some sense one cannot expect an inclusion stronger than $SBD^p(\Omega)\subset SBV(\Omega;\Rn)$, to be precise one cannot hope for a higher integrability of the gradient, as Example \ref{example} shows. 

To conclude this discussion about the connections between $SBD^p$ and $BV$, we mention the result proved in \cite[Theorem A.1]{cha-gia-pon}.
The authors show that any $u\in SBD(\Omega)$ with $e(u)=0$ $\mathcal{L}^n$-a.e. and $\Hn(J_u)<\infty$ has in fact a precise structure, it is a Caccioppoli-affine function. Thanks to the result proved in Section \ref{sbd-caccioppoli} we thus infer that it has actually (generalized) bounded variation. 

\section*{Acknowledgments} 
This work was partially supported 
by the Deutsche Forschungsgemeinschaft through the Sonderforschungsbereich 1060 
{\sl ``The mathematics of emergent effects''}, project A6.

F.~Iurlano thanks the University of Florence for
the warm hospitality of the DiMaI ``Ulisse Dini'', where part of this work was 
carried out during the winter of 2015. 

Part of this work was conceived when M.~Focardi was visiting the University of Bonn in winter 2015. 
He thanks the Institute for Applied Mathematics  for the hospitality and the stimulating scientific 
atmosphere provided during his stay.

M.~Focardi and F.~Iurlano are members of the Gruppo Nazionale per
l'Analisi Matematica, la Probabilit\`a e le loro Applicazioni (GNAMPA)
of the Istituto Nazionale di Alta Matematica (INdAM).

\bibliographystyle{siam}
\bibliography{sbd10}

\end{document}